\documentclass[11pt,leqno,draft]{article}
\usepackage{bbm}
\usepackage{amsfonts}
\usepackage{stmaryrd}
\usepackage{amsmath, amsthm, amsfonts, amssymb, color}
\usepackage{mathrsfs}
\usepackage{color}
\usepackage[notcite,notref]{showkeys}

\newtheorem{thm}{Theorem}[section]

\newtheorem{lem}[thm]{Lemma}

\theoremstyle{definition}

\newcommand{\scr}[1]{\mathscr #1}
\definecolor{wco}{rgb}{0.5,0.2,0.3}

\numberwithin{equation}{section} \theoremstyle{remark}

\newcommand{\ua}{\uparrow}

\title{{\bf Order Preservation and  Positive Correlation for Nonlinear Fokker Planck Equations }  }
\author{
{  Panpan Ren }\\
\footnotesize{Department of Mathematics, University of Bonn, Bonn 53115, Germany}\\
\footnotesize{RPPZOE@GMAIL.COM}\\
}
\begin{document}
\allowdisplaybreaks
\def\R{\mathbb R}  \def\ff{\frac} \def\ss{\sqrt} \def\B{\mathbf
B}
\def\N{\mathbb N} \def\kk{\kappa} \def\m{{\bf m}}
\def\ee{\varepsilon}\def\ddd{D^*}
\def\dd{\delta} \def\DD{\Delta} \def\vv{\varepsilon} \def\rr{\rho}
\def\<{\langle} \def\>{\rangle}
  \def\nn{\nabla} \def\pp{\partial} \def\E{\mathbb E}
\def\d{\text{\rm{d}}} \def\bb{\beta} \def\aa{\alpha} \def\D{\scr D}
  \def\si{\sigma} \def\ess{\text{\rm{ess}}}\def\s{{\bf s}}
\def\beg{\begin} \def\beq{\begin{equation}}  \def\F{\scr F}
\def\Ric{\mathcal Ric} \def\Hess{\text{\rm{Hess}}}
\def\e{\text{\rm{e}}} \def\ua{\underline a} \def\OO{\Omega}  \def\oo{\omega}
 \def\tt{\tilde}\def\[{\lfloor} \def\]{\rfloor}
\def\cut{\text{\rm{cut}}} \def\P{\mathbb P} \def\ifn{I_n(f^{\bigotimes n})}
\def\C{\scr C}      \def\aaa{\mathbf{r}}     \def\r{r}
\def\gap{\text{\rm{gap}}} \def\prr{\pi_{{\bf m},\varrho}}  \def\r{\mathbf r}
\def\Z{\mathbb Z} \def\vrr{\varrho} \def\lll{\lambda}
\def\L{\scr L}\def\Tt{\tt} \def\TT{\tt}\def\II{\mathbb I}
\def\i{{\rm in}}\def\Sect{{\rm Sect}}  \def\H{\mathbb H}
\def\M{\mathbb M}\def\Q{\mathbb Q} \def\texto{\text{o}} \def\LL{\Lambda}
\def\Rank{{\rm Rank}} \def\B{\scr B} \def\i{{\rm i}} \def\HR{\hat{\R}^d}
\def\to{\rightarrow}\def\l{\ell}\def\iint{\int}\def\gg{\gamma}
\def\EE{\scr E} \def\W{\mathbb W}
\def\A{\scr A} \def\Lip{{\rm Lip}}\def\S{\mathbb S}
\def\BB{\scr B}\def\Ent{{\rm Ent}} \def\i{{\rm i}}\def\itparallel{{\it\parallel}}
\def\g{{\mathbf g}}\def\Sect{{\mathcal Sec}}\def\T{\mathcal T}\def\BB{{\bf B}}
\def\f{\mathbf f} \def\g{\mathbf g}\def\BL{{\bf L}}  \def\BG{{\mathbb G}}
\def\Bd{{D^E}} \def\BdP{D^E_\phi} \def\Bdd{{\bf \dd}} \def\Bs{{\bf s}} \def\GA{\scr A}
\def\Bg{{\bf g}}  \def\Bdd{{\bf d}} \def\supp{{\rm supp}}\def\div{{\rm div}}
\def\ddiv{{\rm div}}\def\osc{{\bf osc}}\def\1{{\bf 1}}\def\BD{\mathbb D}\def\GG{\Gamma}
\maketitle

\renewcommand{\bar}{\overline}
\renewcommand{\tilde}{\widetilde}
\maketitle

\begin{abstract} By investigating McKean-Vlasov SDEs, the order preservation and positive correlation are characterized for nonlinear Fokker-Planck equations. The main results recover
the corresponding criteria on these properties established in \cite{Wang, Pitt} for diffusion processes or linear Fokker-Planck equations. 
\end{abstract} \noindent
 \vskip 2cm

 \section{Introduction}

Based on \cite{Pitt}, complete criteria have been established in \cite{Wang} for the order preservation and positive correlation for diffusion processes corresponding to linear Fokker-Planck equations, where the order preservation links to comparison theorem in the literature of SDEs, and the positive correlation arises from statistics is   known as FKG inequality due to \cite{FKG}.  In the present paper we aim to extend these criteria to nonlinear Fokker-Planck equations associated with McKean-Vlasov SDEs.  In the following we first recall the criteria of \cite{Wang}.   

Consider the following second order differential operator on $\R^d$:
$$L=\sum_{i,j=1}^d a_{ij}\partial_i\partial_j+\sum_{i=1}^d b_i\partial_i, $$
where 
$$a=(a_{ij})_{1\le i,j\le d}: \R^d \to \R^d\otimes \R^d,\ \ b=(b_i)_{1\le i\le d}: \R^d\to\R^d$$ are continuous, $a$ is positive definite, such that the martingale problem of $L$ is well-posed, or equivalently, there exists a unique  $L$-diffusion process for any initial distribution.   For any $t\ge 0$ and $\mu\in \scr P$,  the space of all  probability measures on $\R^d$ equipped with the weak topology, let $P^*_t\mu$ be the distribution of the $L$-diffusion  process at time $t$  with   initial distribution $\mu$.

We denote $x\le y$ for $x:=(x_i)_{1\le i\le d},  ~y:=(y_i)_{1\le i\le d}\in \R^d$ if $x_i\le y_i$ for any $1\le i\le d.$ 
Let $\B_b(\R^d)$   be the set of all bounded measurable functions on $\R^d$. Consider  the class of bounded measurable increasing functions:
$$\scr U_b:=\big\{f\in\B_b(\R^d)\big|  f(x)\le f(y) \ \text{for}\  x,y\in\R^d\ \text{with}\ x\le y\big\},$$
and the family of probability measures of positive correlations: 
$$\scr P_{+}:=\big\{\mu\in\scr P|\, \mu(fg)\geq \mu(f)\mu(g)~~f,g\in \scr U_b\big\},$$ where we call $\mu$ satisfies the FKG inequality if $\mu\in \scr P_+$. 
Moreover, we write $\mu \preceq \nu$ for any two probability measures $\mu,\nu\in \scr P,$   if $\mu(f)\leq \nu(f)$ holds  for any  $f\in \scr U_b(\R^d).$ 
Note that  definitions of $\scr P_+$ and $\mu \preceq \nu$ does not change if we replace $\B_b(\R^d)$ by $C_b^k(\R^d)$ for $k\in \mathbb Z_+\cup\{\infty\}$, where $C_b^0(\R^d)=C_b(\R^d)$ denotes the set of all bounded continuous functions on $\R^d$, while when $k\ge 1$ the class  $C_b^k(\R^d)$ consists of bounded functions on $\R^d$  having bounded derivatives up to order $k$.

Now, for the $L$-diffusion process, we denote $P_t^*\in \scr \scr P_t$ if $ P_t^* \scr P_+\subset \scr P_+$.  In this case we call $P_t^*$ preserves  positive correlations. 
 The following result is implied by  \cite[Theorem 1.3 and Theorem 1.4]{Wang}.  

\beg{thm}[\cite{Wang}]\label{CW1}      $P_t^*\in \scr P_+$   for any $t\ge 0$ if and only if the following conditions hold: 
 \beg{enumerate}
\item[$(a)$]  For any $1\le i,j\le d,$ $a_{ij}\ge 0,$  and $a_{ij}(x)$ depends only on $x_i$ and $x_j.$  
\item[$(b)$]  For any $1\le i\le d$, $b_i(x)$ is increasing in $x_j$ for $j\neq i.$
\end{enumerate}
\end{thm}

Next, let $\bar P_t^*$ be associated with the diffusion process generated by another  operator $\bar L$ satisfying the same assumption on $L$:
$$\bar L=\sum_{i,j=1}^d \bar a_{ij}\partial_i\partial_j+\sum_{i=1}^d \bar b_i\partial_i.$$
We denote $\bar P_t^* \preceq  P_t^*$ if 
$$\bar P_t^*\mu  \preceq    P_t^* \nu,   \ \  \mu,\nu\in \scr P, \mu \preceq \nu.$$
In this case we call these two diffusion processes (or $\bar P_t^*$ and $P_t^*$) order-preserving, and when $L=\bar L$, we call the $L$-diffusion process (or $P_t^*$) monotone. 
The next result follows from  \cite[Theorem 1.3]{Wang}.

\beg{thm}[\cite{Wang}]\label{CW2}      $\bar P_t^* \preceq  P_t^* $   for any $t\ge 0$,  if and only if the following conditions hold:  
 \beg{enumerate}
\item[$(a)$]  For any $1\le i,j\le d,$ $\bar a_{ij}=a_{ij},$  and $a_{ij}(x)$ depends only on $x_i$ and $x_j.$  
\item[$(b)$]  For any $1\le i\le d$, $\bar b_i(x)\le b_i(y) $ for $x\le y$ with $x_i=y_i$.  
\end{enumerate}
\end{thm}
According to Theorem \ref{CW1} and Theorem \ref{CW2} with $L=\bar L$, we see that $P_t^*\in \scr P_+$ implies the monotonicity. 

To conclude this subsection, we  recall the link of  diffusion processes and   linear Fokker-Planck equations. 
By the definition of martingale problem,  we see that    $\mu_t:=P_t^*\mu$ solves the following linear Fokker-Planck equation on $\scr P$:
\beq\label{LFK} \pp_t\mu_t= L^*\mu_t,\end{equation} in the sense that $\mu\in C([0,\infty);\scr P)$ satisfies 
$$\mu_t(f):=\int_{\R^d} f\d\mu_t= \mu_0(f) +\int_0^t\mu_s(Lf)\d s,\ \ t\ge 0, f\in C_0^{\infty}(\R^d).$$
On the other hand, according to the superposition principle, see   \cite{TR, SV},  a solution $\mu_t$ of of \eqref{LFK} with
$$\int_0^t \mu_s\big(\|a\|+|b|\big)\d s <\infty,\ \ t\ge 0$$
is given by $\mu_t:= P_t^*\mu_0, t\ge 0.$ Since the order preservation and positive correlation  are distribution properties of   diffusion processes,  
they are indeed properties of solutions to linear Fokker-Planck equations. 

As mentioned above, we aim to  extend Theorems \ref{CW1}  and \ref{CW2}  to nonlinear Fokker-Planck equations.  Consider the Wasserstein space 
$$\scr P_2=\{\mu\in \scr P: \mu(|\cdot|^2)< \infty \},$$
which is a Polish space under the Wasserstein distance 
$$\W_2(\mu,\nu):=\inf_{\pi\in \C(\mu,\nu)} \bigg(\int_{\R^d\times \R^d}|x-y|^2\pi(\d x,\d y)\bigg)^{\ff 1 2},$$
where $\C(\mu,\nu)$ is the space of all couplings of $\mu$ and $\nu$. 
Consider the following time-distribution dependent second order differential operators
\beq\label{N2} \beg{split} &L_{t,\mu}  
:=\sum_{i,j=1}^d a_{ij}(t,\cdot,\mu)\partial_i\partial_j+\sum_{i=1}^d b_i(t,\cdot, \mu)\partial_i,\\
& \bar L_{t,\mu}  
:=\sum_{i,j=1}^d \bar a_{ij}(t,\cdot,\mu)\partial_i\partial_j+\sum_{i=1}^d \bar b_i(t,\cdot, \mu)\partial_i,\end{split} \end{equation} 
where
\beg{align*} &a=(a_{ij})_{1\le i,j\le d},\ \bar a=(\bar a_{ij})_{1\le i,j\le d}: [0,\infty)\times\R^d\times \scr P_2\to \R^d\otimes \R^d,\\
&b=(b_{i})_{1\le i\le d},\ \bar b=(\bar b_{i})_{1\le i\le d}: [0,\infty)\times\R^d\times \scr P_2\to \R^d \end{align*} are continuous. 
The nonlinear Fokker-Planck equations for $L$ and $\bar L$ are formulated as 
\beq\label{NFP} \pp_t \mu_t= L_{t,\mu_t}\mu_t,\ \ \pp_t \bar \mu_t= \bar L_{t,\bar\mu_t}\bar \mu_t,\ \ \mu_s=\bar\mu_s=\mu, t\ge s.\end{equation}
We call $(\mu_t,\bar\mu_t)_{t\ge s}  \in C([s,\infty);\scr P_2)\times C([s,\infty);\scr P_2)$ a solution to \eqref{NFP}, if
 \beg{align*} & \mu_t(f)= \mu(f)+\int_s^t \mu(L_{r,\mu_r}f)\d r,\\
 &\bar \mu_t(f)= \mu(f)+\int_s^t\bar \mu(\bar L_{r,\bar \mu_r}f)\d r,\ \ t\ge s, f\in C_0^\infty(\R^d).\end{align*}
The nonlinear Fokker-Planck equations can be characterized by distribution dependent SDEs, also called McKean-Vlasov or mean field SDEs. 
Let $W_t$ be a $d$-dimensional Brownian motion on a complete filtration probability space $(\OO, \{\F_t\}_{t\ge }, \P)$. Consider the distribution dependent SDEs
\beq\label{DDSDE} \beg{split} &\d X_t= b(t,X_t,\L_{X_t})\d t +\ss{2 a(t,X_t,\L_{X_t})}\d W_t,   \\
 & \d \bar X_t= \bar b(t,\bar X_t,\L_{\bar X_t})\d t +\ss{2 \bar a(t,\bar X_t,\L_{\bar X_t})}\d W_t,\end{split}\end{equation} 
  where $\L_{\xi}$ denotes the distribution of a random variable $\xi$.
We call these SDEs well-posed if for any $s\ge 0$ and any $\F_s$-measurable random variables $X_s$ and $\bar X_s$ with $\E [|X_s|^2+|\bar X_s|^2]<\infty,$
they have unique solutions with $(\L_{X_t})_{t\ge s}, (\L_{\bar X_t})_{t\ge s}\in C([s,\infty);\scr P_2).$ In this case, we denote
$$P_{s,t}^*\mu:=\L_{X_t},\  \bar P_{s,t}^*\mu:=\L_{\bar X_t}, ~\L_{X_s}=\L_{\bar X_s}=\mu\in \scr P_2,\ \ t\ge s.$$  
By It\^o's formula, $(\mu_t:= P_{s,t}^*\mu, \bar\mu_t:= \bar P_{s,t}^*\mu)$ solves \eqref{NFP}. On the other hand, by the superposition principle, see \cite{BR}, if $(\mu_t,\bar\mu_t)$ solves \eqref{NFP} with
$$\int_s^t \mu_r\big(|b(r,\cdot,\mu_r)|+\|a(r,\cdot,\mu_r)\|\big)\d r+ \int_s^t \bar\mu_r\big(|\bar b(r,\cdot,\mu_r)|+\|\bar a(r,\cdot,\bar \mu_r)\|\big)\d r<\infty,\ \ t\ge s,$$
then $\mu_t= P_{s,t}^*\mu$ and $\bar\mu_t=\bar P_{s,t}^*\mu, t\ge 0$.    Thus, to investigate the nonlinear Fokker-Planck equations \eqref{NFP} is equivalent to study the solutions of 
the distribution dependent SDEs \eqref{DDSDE}. 

Unlike in  the setting of standard Markov processes or linear Fokker-Planck equations,  the distribution $\L_{(X_t)_{t\ge s}}$   
on the path space is no longer determined by its time-marginals $P_{s,t}^*\mu$, so we will also consider the order  preservation and positive correlations for  
 $$   \LL_s\mu:= \L_{(X_t)_{t\ge s}},\ \  \ \bar \LL_s\mu=\L_{(\bar X_t)_{t\ge s}},\ \ \ s\ge 0,\mu\in \scr P_2.$$

In Section 2, we state our main results on the order preservation for $(P_{s,t}^*, \bar P_{s,t}^*)$ and $(\LL_s, \bar\LL_s)$, as well as on the positive correlations for $\LL_s$.
To prove these results, In Section  3  we extend Theorems \ref{CW1} and \ref{CW2} to the time inhomogeneous setting which are also new in the literature. 
Finally, the main results are proved in Section 4.

\section{Main results} 

To ensure the well-posedness and to apply the superposition principle, we make the following assumption. 

 \beg{enumerate}
\item[({\bf A})] $b,\bar b, a, \bar a$ are continuous on $[0,\infty)\times\R^d\times\scr P_2$, $a$ and $\bar a$ are positive definite,   
$$ \mu\big(|b(t,\cdot,\mu)|+\|a(t,\cdot,\mu)\|\big):=\int_{\R^d} \big(|b(t,\cdot,\mu)|+\|a(t,\cdot,\mu)\|\big)\d\mu$$ is locally bounded in $(t,\mu)\in [0,\infty)\times \scr P_2$, and 
 there exits an increasing function $K:[0,\infty)\to [0, \infty)$ such that $b, \bar b, \si:=\ss{2a}$ and $\bar\si:=\ss{2\bar a}$ satisfy 
\beq\label{A1}
\begin{split}
\max\Big\{&2\<b(t,x,\mu)-b(t, y,\nu), x-y\>+\|\si(t,x,\mu)-\si(t,x,\nu)\|_{HS}^2,\\
&2\<\bar b(t,x,\mu)-\bar b(t, y,\nu), x-y\>+\|\bar \si(t,x,\mu)-\bar \si(t,x,\nu)\|_{HS}^2\Big\}\\
&\leq K(t)(|x-y|^2+\W_2(\mu,\nu)^2), ~~~t\ge 0,~~x,y\in \R^d,~~\mu,\nu \in \scr P_2.
\end{split}
\end{equation}
\end{enumerate}
  According to \cite{Wang1} and the superposition principle in \cite{BR}, {\bf (A)} implies the well-posedness of \eqref{DDSDE}  and  \eqref{NFP}, and 
for any $s\ge0, \mu\in \scr P_2$, 
$$\mu_t:=P_{s,t}^*\mu,\ \ \ \bar \mu_t:= \bar P_{s,t}^*\mu,\ \ t\ge s$$ give the unique solution of \eqref{NFP}. 

To state the main results, we first define the order preservation  and positive correlations in the present setting.
 For any 
 $$\xi,\eta\in C_s:=C([s,\infty); \R^d), $$ we denote $\xi\le \eta$ if $\xi_t\le \eta_t$ for all $t\ge s.$ For any two probability measures $\Phi_1,\Phi_2$ on the path space $C_s$, we denote
  $\Phi_1 \preceq \Phi_2$ if $\Phi_1(F)\le \Phi_2(F)$ holds for any bounded increasing function $F$ on $C_s$. 
  Similarly, let $\scr P_+^s$ denote the set of probability measures on $C_s$ satisfying the FKG inequality for bounded increasing functions on $C_s$. 
  
\beg{defn} Let $t\ge s\ge 0.$
\beg{enumerate} \item[(1)] We write $\bar P_{s,t}\preceq P_{s,t}^*$, if $\bar P_{s,t}^*\mu \preceq P_{s,t}^*\nu$ holds for any   $\mu,\nu\in \scr P_2$ with $\mu\preceq \nu.$
\item[(2)] We write $\bar \LL_s\preceq \LL_s$, if $\bar \LL_s\mu\preceq \LL_s\nu$ holds for any  $\mu,\nu\in \scr P_2$ with $\mu\preceq \nu.$ 
\item[(3)] We write $P_{s,t}^*\in \scr P_+$ if $P_{s,t}^*\scr P_+\subset \scr P_+$;  and $\LL_s \in \scr P_+^s$ if $\LL_s\mu\in \scr P_+^s$  holds for all $\mu\in \scr P_+.$ 
\end{enumerate}
\end{defn} 
Obviously, $\bar \LL_s \preceq \LL_s$  for all $s\ge 0$ implies $\bar P_{s,t}\preceq P_{s,t}^*$ for all $t\ge s\ge 0$, but the inverse may not true in the nonlinear setting. Similarly, 
$\LL_s\in \scr P_+^s$ implies $P_{s,t}^*\in \scr P_+$ for any $t\ge s$ but the inverse may not be true. 

\subsection{Order preservation} 

The following result provides sufficient conditions for the order preservation. 

\beg{thm}\label{T1} Assume {\bf (A)} and the following two conditions:
\beg{enumerate} \item[$(1)$]   For any $1\le i\le d$ and $s\ge 0$,  $\bar b_i(s, x, \nu) \leq b_i(s,y, \mu)$ holds for $x\preceq y $  with  $ x_i=y_i$ and $ \nu\preceq \mu;$
\item[$(2)$]  $a=\bar a$,     and for any $1\le i,j\le d, s\ge 0$ and $\mu\in\scr P_2$, $ a_{ij}(s,x,\mu)$ depends only on $x_i$ and $x_j.$
\end{enumerate}
Then $\bar \Lambda_s \preceq \Lambda_s$ for all $s\ge 0$. Consequently,   $\bar P^*_{s,t}\preceq P^*_{s,t}$ for $t\geq s.$
 \end{thm}    

The next two results include necessary conditions for the order preservation, which are weaker than the sufficient ones given in Theorem \ref{T1}. 
However, they coincide with the sufficient conditions and hence become sufficient and necessary conditions when $b(t,x,\mu)$ and $a(t,x,\mu)$ do not depend on $\mu$, and hence our first three results recover Theorems \ref{CW1} and \ref{CW2}. 

For any $\mu\in \scr P$ and $I\subset \{1,\cdots,d\},$  let 
$$\mu_I(A):= \mu(\{x\in \R^d: x_I\in A\}),\ \ A\in \B(\R^{\# I})$$   be  the marginal distribution of $\mu$ with respect to   components indexed by  $I$, where $\# I$ denotes the number of elements in $I$.  In particular, we simply denote $\mu_i=\mu_{\{i\}}$ 

\beg{thm}\label{T2}  
If  $\bar \Lambda_s\preceq  \Lambda_s$ for all $s\geq 0$, then the following conditions hold:
\beg{enumerate}
\item[$(i)$] for any   $\nu \preceq \mu$ with $\nu_i=\mu_i,$ $1\leq i \leq d$, there exists a coupling $\pi \in \C(\nu,\mu)$ with $\pi \big(\{x\leq y\}\big)=1$  such that $$\bar b_i(s, x, \nu) \leq b_i(s,y, \mu),~~~s\ge 0,~(x,y)\in \supp\pi.$$
Consequently,  $\bar b_i(s,x,\mu) \leq b_i(s,x,\mu) $  for $s\ge0,x\in\R^d,\mu\in\scr P_2.$
\item[$(ii)$] for any $ \nu\preceq \mu$ with $\nu_{ij}=\mu_{ij},1\leq i,j \leq d,$ there exists  a coupling $\pi \in \C(\nu,\mu)$ with $\pi \big(\{x\leq y\}\big)=1$  such that $$\bar a_{ij}(s, x, \nu)=a_{ij}(s,y,\mu),~s\geq0, ~(x,y)\in \supp\pi.$$
Consequently,   $a(s,x,\mu)=\bar a(s,x,\mu)$ for any $s\ge0,x\in\R^d,\mu\in\scr P_2.$
\end{enumerate}
\end{thm}

Since $\bar \Lambda_s\preceq   \Lambda_s$ implies $\bar P_{s,t}^* \preceq P_{s,t}^*$ for $t\ge s$, conditions in the following result are also necessary for $\bar \Lambda_s\preceq   \Lambda_s$.

\begin{thm}\label{T3}  
If $\bar P_{s,t}^* \preceq  P_{s,t}^*~\mbox{for}~t\geq s\ge 0,$ then the following conditions hold:
\beg{enumerate}
\item[$(i)$] For any $s\ge 0$ and $1\le i\le d$, $\nu(\bar b_i(s,\cdot, \nu)) \leq \mu(b_i(s,\cdot, \mu))$ holds for $\nu\preceq \mu$  with $\nu_i=\mu_i.$
\item[$(ii)$] For any $s\ge 0$ and $1\le i,j\le d$, $\bar a_{ij}(s, x, \delta_x)=a_{ij}(s, x, \delta_x)$ holds and $a_{ij}(x, \delta_x)$ depends only on $x_i$ and $x_j.$
\end{enumerate}
\end{thm}

\subsection{Positive correlations} 
We first present sufficient conditions for the positive correlations. 
 \begin{thm}\label{T4}
Assume {\bf (A)} and suppose further 
\beg{enumerate}
\item[$(1)$] For any $s\ge 0$ and $1\le i\le d$,  $b_i(s,x, \nu)\leq b_i(s,y, \mu),~ \nu\preceq \mu, x\leq y~\mbox{with}~x_i=y_i.$

\item[$(2)$]  For any $1\le i,j\le d$, $a_{ij}\geq 0$, and  for  any  $\mu \in \scr P_+,$   $a_{ij}(s,x,\mu)$ depends only on $x_i$ and $x_j.$
\end{enumerate}
Then $\Lambda_s \in \scr P_+^s.$ Consequently, $P^*_{s,t}\in \scr P_+$ for any $t\geq s\geq 0. $  
\end{thm}

 Similarly to the above results on the order preservation, necessary conditions for positive correlations presented in the next result are weaker than the above sufficient ones, but they coincide 
and hence become necessary and sufficient conditions  when $b$ and $a$ do not depend on the distribution.
 
 \beg{thm}\label{T5}  If  $  \Lambda_s^{\mu} \in \scr P_+^s$ for  $s\ge0$ and  $\mu=\mu_{\{ij\}}\times\mu_{\{ij\}^c}\in\scr P_+$, then  the following assertions hold:
\beg{enumerate}
\item[$(1)$]  For any $s\ge 0$, $1\le i,j\le d$, $a_{ij}(s,x,\mu)$ depends only on $x_i$ and $x_j.$
\item[$(2)$] For any $s\ge 0$, $1\le i\le d$ and $f \in \scr U_b$    independent on $x_i,$
$$\mu(b_i(s,\cdot, \mu) f) \geq \mu(f)\mu(b_i(s,\cdot,\mu)).$$
\end{enumerate}
\end{thm}

\section{Time-inhomogeneous  diffusion processes}

Consider the time-dependent second order diffusion operators: for $t\geq 0$ and $  x\in\R^d$,
\begin{equation}\label{OP}
\begin{split}
L_t:&=\frac{1}{2}\sum_{i,j=1}^d a_{ij}(t,x)\pp_i\pp_j+\sum_{i=1}^d b_i(t,x)\pp_i, \\
\bar L_t:&=\frac{1}{2}\sum_{i,j=1}^d \bar a_{ij}(t,x)\pp_i\pp_j+\sum_{i=1}^d\bar b_i(t,x)\pp_i,
\end{split}
\end{equation}
where $a_{ij},\bar a_{ij},b_i,\bar b_i\in C([0,\infty)\times\R^d)$. Assume that the  martingale problems associated with $(L_t)_{t\ge0}$ and  $(\bar L_t)_{t\ge0}$  are well-posed
so that there exist  unique time-inhomogeneous diffusion processes $(X_{s,t})_{t\ge s\ge0}$ and  $(\bar X_{s,t})_{t\ge s\ge0}$  corresponding to $(L_t)_{t\ge0}$ and  $(\bar L_t)_{t\ge0}$, respectively.
Let $(P_{s,t})_{t\geq s\ge0}$  and $(\bar P_{s,t})_{t\geq s\ge0}$ be the  Markov semigroups generated by $(X_{s,t}^x)_{\{ij\}}\times\mu_{\{ij\}^c}$ and  $(\bar X_{s,t}^x)_{\{ij\}}\times\mu_{\{ij\}^c}$ with the initial value  $X_{s,s}=\bar X_{s,s}=x$, respectively, i.e.,
\begin{equation}\label{semi}
P_{s,t}f(x)=\E f(X_{s,t}^x),\, \bar P_{s,t}f(x)=\E f(\bar X_{s,t}^x), \, \quad f\in \scr B_b(\R^d).
\end{equation}
It is well known that for any $f \in C_0^\infty(\R^d)$ 
\begin{equation}\label{PP}
\frac{\d}{\d s} P_{s,t}f(x)= -P_{s,t} L_sf,~~~~~\frac{\d}{\d t} P_{s,t}f(x)=  L_tP_{s,t}f,\quad t\ge s\ge0. 
\end{equation}
 For any $x,y\in\R^d$ with $x\leq y $, $f \in \scr U_b$ and $t\ge s\ge0$, if $\bar P_{s,t}f(x)\leq P_{s,t}f(y),$  we call $P_{s,t}$  preserving order,  written as $\bar P_{s,t}^*\preceq P_{s,t}^*,$
 where for any $\mu\in \scr P$, $P_{s,t}^*\mu, \bar P_{s,t}^*\mu\in \scr P$ is given by 
 $$(P_{s,t}^*\mu)(f):= \mu(P_{s,t}f),\ \ (\bar P_{s,t}^*\mu)(f):= \mu(\bar P_{s,t}f),\ \   f\in \B_b(\R^d).$$
 Moreover, we denote $P_{s,t}^*\in\scr P_+$ if $P_{s,t}^*\scr P_+\subset \scr P_+$. 

For any $\mu\in \scr P$, let $\LL_s\mu$ and $\bar \LL_s\mu$ be the distributions  of  the  processes starting at $\mu$ from time $s$ generated by $L$ and $\bar L$ respectively.
By the standard Markov property we see that $\bar P_{s,t}^*\preceq  P_{s,t}^*$ for $t\ge s\ge 0$ if and only if $\bar\LL_s\preceq \LL_s$ for $s\ge 0$, while $P_{s,t}^*\in \scr P_+$ for $t\ge s$ is equivalent to $\LL_s\in \scr P_+^s.$ 

 \subsection{Main results} 
 
The following two results extend  Theorems \ref{CW1} and  \ref{CW2} to the present time-inhomogeneous setting. 

 \beg{thm}\label{T1.3}   If $\bar P_{s,t}^* \preceq  P_{s,t}^* $ for $ t\geq s\ge0$, equivalently $\bar \LL_s\preceq \LL_s$ for $s\ge 0$,  if and only if the following conditions hold: 
  \beg{enumerate}
 \item[$({\bf 1})$] For any $s\ge 0$ and  $1\le i\le d$,  
  $\bar b_i(s,x)\leq b_i(s,y) $ with $x\leq y$ and $ x_i=y_i.$
\item[$({\bf 2}$)]  For any $s\ge 0$ and  $1\le i,j\le d$,  $ \bar a_{ij}=a_{ij}$ and $a_{ij}(s,x)$ only depends on $x_i$ and $x_j$.
\end{enumerate}
\end{thm}

\beg{thm}\label{T1.6} $P_{s,t}^*\in \scr P_+$ for $t\ge s$, equivalently $\LL_s\in \scr P_+^s$ for $s\ge 0$, if and only if the following conditions hold:
 \beg{enumerate}
 \item[$(1)$] For any $s\ge 0$ and $1\le i\le d$,  $b_i(s, x)\leq b_i(s, y) $  with   $x\leq y$ and  $x_i=y_i; $
\item[$(2)$]  For any $s\ge 0$ and $1\le i\le d$, $a_{ij}\ge 0$ and $ a_{ij}(s,x)\ge 0$ depends only on $x_i$ and $x_j.$
\end{enumerate}
\end{thm}

\subsection{Proofs} 

\beg{proof}[Proof of Theorem \ref{T1.3}]  (a) We first prove the  necessity. 
 For any $t\ge s\ge0$ and $x\in\R^d$,  let $\Lambda _{s}^x$ (resp. $\bar \Lambda _{s}^x$) be the distribution of the $L_t$-diffusion (resp. $\bar L_t$-diffusion) process on the path space $C_s:=C([s,\infty);\R^d)$ starting from $x$ at time $s.$

For   $x\in\R^d$ and $0\le s_0\leq s_1< s_2<\cdots <s_n$,  let $\Lambda^x_{s_0, s_1,\cdots, s_n}$ be the marginal distribution of $\Lambda_{s_0}^x$ at the time sequence $(s_1,\cdots, s_n)$, which   can be expressed via the Markov property as below
$$\Lambda^x_{s_0, s_1,\cdots, s_n}(\d y_1,\d y_2,\cdots,\d y_n)=P_{s_0,s_1}(x,\d y_1)P_{s_1,s_2}(y_1,\d y_2)\cdots P_{s_{n-1},s_n}(y_{n-1},\d y_n).$$
Then, by an inductive argument, together with the Markov property of the associated Markov process,   $\bar P_{s,t}^*\preceq P_{s,t}^*$ implies $\bar \Lambda _{s}^x\leq \Lambda _{s}^y$ (i.e., $\bar \Lambda_{s }^x(f)\leq  \Lambda_{s }^y(f)$ for any $f \in \scr U_b\cap C_s$. Therefore,
there exists a coupling $\P_{s }^{x,y} \in \C (\bar \Lambda^x_{s }, \Lambda^y_{s })$ such that
\begin{equation}\label{D}
\P_{s }^{x,y} \big( (\xi,\eta)\in C_{s }\times C_{s }: \eta \preceq \xi \big)=1.
\end{equation}
Let $(\Omega,\F,\P)=(C_{s }\times C_{s },\B(C_{s }\times C_{s} ),\P_{s }^{x,y})$ with the natural filtration $(\F_t)_{t\geq s }$ induced by the coordinate process $(\xi_t,\eta_t)_{t\geq s } $  solving
\begin{equation}\label{coor}
  \begin{cases}
   \d \xi_t=b(t,\xi_t)\d t+\si(t, \xi_t) \d B_t^1,\quad      & \xi_{s}=y\\
    d \eta_t=\bar b(t,\eta_t)\d t+\bar \si(t, \eta_t) \d B_t^2,\quad     &  \eta_{s}=x
  \end{cases}
\end{equation}
for some $d$-dimensional Brownian motions $(B_t^1)_{t\geq s}$ and $(B_t^2)_{t\geq s},$ and some measurable mappings $\si, \bar \si: [s,\infty)\times \R^d \rightarrow \R^d \otimes \R^d$ with
 $a=\si \si^*,~\bar a=\bar \si\, \bar \si^*$. Then, from \eqref{D}, we have
$\xi_t\geq \eta_t$,  $\P_{s }^{x,y}$-a.s., for all $t\ge s$.

 Let $x\leq y$ with $x_i=y_i$. Since    $\xi_t\geq \eta_t$, $\P_{s}^{x,y}$-a.s., and  $(\xi_s)_i=(\eta_s)_i$ due to $x_i=y_i$, we derive from \eqref{coor} that
$$\int^t_{s} \big(b_i(r,\xi_r)-\bar b_i(r,\eta_r)\big)\d r\geq \int^t_{s} \big\<\bar \si_{i\cdot}(r,\xi_r), \d B_r^2\big\>-\int^t_{s} \big\<\si_{i\cdot}(r,\xi_r), \d B_r^1\big\>,$$
where $\si_{i\cdot}$ means the $i$-th row of $\si.$
Taking conditional expectation $\P_{s_0}^{x,y}(\cdot| \F_{s_0})$ on both sides yields
$$\int^t_{s } \E \big(\big(b_i(r,\xi_r)-\bar b_i(r,\eta_r)\big)|\F_{s }\big)\d r\geq 0,~~~t\geq s.$$
This implies  the assertion ({\bf 1}) by taking the continuity of
$b_i,~\bar b_i$ and $(\xi_\cdot,~\eta_\cdot)$ into account.

Let $x\leq y$  with $(x_i,x_j)=(y_i, y_j)$. Then, by using  $\xi_t\geq \eta_t, \P_{s_0}^{x,y}$-a.s., again,  we have
\begin{equation}\label{coor1}
   \int_{s_0}^t b_k(s,\xi_s)\d s +\int_{s_0}^t \big\<\si_{k\cdot}(s,\xi_s), \d B_s^1\big\>\geq    \int_{s_0}^t \bar b_k(s,\eta_s)\d s +\int_{s_0}^t \big\<\bar \si_{k\cdot}(s,\eta_s), \d B_s^2\big\>,\, k=i,j.
\end{equation}
Note that as $t\downarrow s_0,$
$$\frac{1}{\sqrt{t-s_0}}\Big( \int_{s_0}^t\big\<\si_{i\cdot}(s,\xi_s), \d B_s^1\big\>, \int_{s_0}^t\big\<\si_{j\cdot}(s,\xi_s), \d B_s^1\big\> \Big)\overrightarrow{weakly} ~~N\left(
      0,
      \left(\begin{array}{c}
      a_{ii}(s, y)~~  a_{ij}(s,y) \\
      a_{ji}(s, y)~~  a_{jj}(s,y)
    \end{array}\right)
   \right) =:\mu,$$
and that
$$\frac{1}{\sqrt{t-s_0}}\Big( \int_{s_0}^t\big\<\bar \si_{i\cdot}(s,\xi_s), \d B_s^2\big\>, \int_{s_0}^t\big\<\bar \si_{j\cdot}(s,\xi_s), \d B_s^1\big\> \Big)\overrightarrow{weakly} ~~N\left(
      0,
     \left( \begin{array}{c}
      \bar a_{ii}(s, y)~~  \bar a_{ij}(s,y) \\
      \bar a_{ji}(s, y)~~  \bar a_{jj}(s,y)
    \end{array}\right)
   \right) =:\bar \mu.$$
Then \eqref{coor1} implies $\bar \mu\preceq  \mu$. On the other hand, by the symmetry of  $\mu$ and $\bar \mu$ due to the symmetry of $a$ and $\bar a$,   then $\bar \mu\preceq \mu $ implies  $\bar \mu \preceq \mu.$ Therefore, we have $\mu=\bar \mu $ so that  $a=\bar a$.  For the assertion that $a_{ij}$ depends only on $x_i$ and $x_j$ of 
({\bf 2}),
it can be  available by following exactly the arguments of \cite[Lemmas 2.1 \& Lemma 2.3]{Wang}.

(b) Following exactly the arguments of \cite[Lemmas 2.4,  2.5 \& Theorem 1.3]{Wang} by replacing  time homogeneous semi-group  $P_t$ by time inhomogeneous semi-group  $P_{s,t}$, we prove the sufficiency  by the following Theorem \ref{T1.4} on the monotonicity.  \end{proof}

\beg{thm}\label{T1.4}    $ P_{s,t}^*$ is monotone, i.e., $P_{s,t}^*\preceq P_{s,t}^*$ for $t\ge s\ge 0$,  provided the following two conditions hold: 
\beg{enumerate}
 \item[$({\bf1'})$] $b_i(s, x)\leq b_i(s, y) $  with   $x\preceq y$ and  $x_i=y_i; $
\item[$({\bf2'})$]  $ a_{ij}(s,x)$ depends only on $x_i$ and $x_j.$
\end{enumerate}
\end{thm}

\begin{proof}
To get $P_{s,t}f \in \scr U_b$ for $t\geq s$ and $f\in \scr U_b$,
it suffices to show $$\nabla P_{s,t} f(x)\geq 0,\quad  t\geq s, \, f\in \scr U_b \cap C_b^{\infty}(\R^d)$$
since $\scr U_b \cap C_b^{\infty}(\R^d)$ is dense in $\scr U_b$.  Below, we assume $f\in \scr U_b \cap C_b^{\infty}(\R^d)$. Let $u_{s,t}=P_{s,t} f,~t\geq s.$ Then by \eqref{PP}, we have
$$\pp_t u_{s,t}=L_t u_{s,t},~~~t\geq s,~~~~~u_{s,s}=f.$$
Taking the partial derivative w.r.t. the $k$-th component (i.e., $\pp_k$) on both sides yields
\beg{equation}\label{Ust}
\pp_t(\pp_k u_{s,t})=\pp_k\pp_t u_{s,t}=L_t^k(\pp u_{s,t})+\sum_{j=1}^d \alpha_{kj}(t,\cdot)\pp_j u_{s,t},
\end{equation}
where $$L_t^k:= A_t+\sum_{j=1}^d\big[(1-\frac{1}{2})\pp_k a_{jk}(t,\cdot)\big]\pp_j+\pp_k b_k(t,\cdot), \quad \alpha_{kj}(t,\cdot):=(\pp_kb_j(t,\cdot))I_{\{k\neq j\}}.$$
Since $L_t^k$ is a  time-inhomogeneous  Schr\"odinger operator, it generates a positivity-preserving semigroup $\big(T_{s,t}^k\big)_{t\geq s}.$ So, the operator $L_t:= \big(L_t^k\big)_{1\leq k \leq d}$ defined on $C^2\big(\R^d;\R^d\big)$ by $L_tV:=\big(L_t^k V_k\big)_{1\leq k \leq d}$ generates a  positivity preserving semigroup$$T_{s,t}:=\big(T_{s,t}^k\big)_{1\leq k \leq d}~,~~~~~t\geq s. $$
Let $D_r=\big(\alpha_{kj}(r,\cdot)_{1\leq k,j\leq d}\big)$ and $V_{s,t}=\nabla P_{s,t}f=\nabla u_{s,t}$. Then $\eqref{Ust}$ implies
\begin{equation*}\label{V}
\pp_t V_{s,t}=L_t V_{s,t}+D_t V_{s,t},~~~ t\geq s, ~~~V_{s,s}=\nabla f.
\end{equation*}
This, together with  Duhamel's formula,  gives
 $$V_{s,t}=T_{s,t}V_{s,s}+\int_s^t T_{r,t} D_r V_{s,r} \mbox{d}r,~~~t\geq s.$$
Thus,   we conclude that  $V_{s,t}=\nabla P_{s,t}f \geq 0$ since $V_{s,s}=\nabla f\geq 0$ due to $f\in \scr U_b$ and $T_{s,t},~D_r$ are positivity preserving.
\end{proof}

\beg{rem}
Different from the proof in \cite{Pitt} for the time-homogeneous setting, we adopt the Duhamel's formula instead of the Trotter product formula which is less explicit in the present setting.
\end{rem}

  \beg{proof}[Proof of Theorem \ref{T1.6}] 
Theorem \ref{T1.6} can be proved  using the same arguments in \cite[Proof of Proposition 4.1]{Pitt} by combining Theorem \ref{T1.3} and  \ref{T1.4}.  So, we omit the details to save space.   \end{proof}  

\section{Proofs of Theorems \ref{T1}, \ref{T2} and \ref{T3}}


\begin{proof}[Proof of Theorem \ref{T1}]
Note that $\bar P_{s,t}^* \nu$  and $ P_{s,t}^* \mu$ are marginal distribution at time $t$ of $\bar \Lambda_s\nu$ and $\Lambda_s\mu$, respectively, we infer that $\bar P^*_{s,t}\preceq P^*_{s,t}$ for $t\geq s $ once $\bar \Lambda_s \preceq \Lambda_s $ is available. Therefore,
to obtain the desired assertion, it is sufficient to show $\bar \Lambda_s \nu \preceq \Lambda_s \mu$. Below, we set $\mu, \nu \in \scr P_2$ with $\nu \preceq \mu$. For any $T>s$,
set
 $$\scr P^{\nu, \mu}_{s,T}:=\big\{ (\mu^{(1)},\mu^{(2)}) \in C([s,T]; \scr P_2\times \scr P_2):\mu_t^{(1)}\preceq \mu_t^{(2)},~t\in [s,T],~\mu_s^{(1)}=\nu, \mu_s^{(2)}=\mu\big\},$$
which is a complete metric space  under the metric for $\lambda>0$,
$$\rr_{\lambda}\big( (\mu^{(1)},\mu^{(2)}), (\tilde \mu^{(1)}, \tilde \mu^{(2)}) \big):=
\sup_{t\in [s,T]} \e^{-\lambda t} \big\{\W_2(\mu^{(1)}_t, \tilde \mu^{(1)}_t)+\W_2(\mu^{(2)}_t, \tilde \mu^{(2)}_t)\big\}.$$
For any $(\mu^{(1)},\mu^{(2)})\in \scr P_{s,T}^{\nu, \mu},$ consider the following time-dependent SDEs:  
\begin{equation}\label{ddsde1}
  \begin{cases}
   \d X_t^{(1),\mu^{(1)}}=\bar b(t, X_t^{(1),\mu^{(1)}}, \mu^{(1)}_t)\d t+ \bar \si(t, X_t^{(1),\mu^{(1)}}, \mu^{(1)}_t) \d W_t     & t\geq s,\quad X_s^{(1),\mu^{(1)}}=\xi\sim\nu,\\
    \d X_t^{(2),\mu^{(2)}}=b(t, X_t^{(2),\mu^{(2)}}, \mu^{(2)}_t)\d t+ \si(t, X_t^{(2),\mu^{(2)}}, \mu^{(2)}_t) \d W_t       &  t\geq s,\quad X_s^{(2),\mu^{(2)}}=\eta\sim\mu,
  \end{cases}
\end{equation} where   $\si=\ss{2a}$ and $\bar\si=\ss{2\bar a}$, and $\xi \sim\nu$ means $\L_{\xi}=\nu$.  
Define the mapping on $\scr P^{\nu, \mu}_{s,T}$  by
\begin{equation}\label{N1}
H\big( (\mu^{(1)},\mu^{(2)})\big)(t)=\Big(\L_{X_t^{(1),\mu^{(1)}}},  \L_{X_t^{(2),\mu^{(2)}}} \Big),\quad t\ge s.
\end{equation}
Since $\mu^{(1)}_t\preceq \mu^{(2)}_t$,  the standard Banach fixed point theorem yields
\begin{equation}\label{N3}
\L_{X_{[s,T]}^{(1),\mu^{(1)}}}   \preceq \L_{X_{[s,T]}^{(2),\mu^{(2)}}},
\end{equation}
In the sequel, we aim to prove that $H$ is contractive under the metric $\rr_{\lambda}$ for large enough $\lambda>0.$
Let $(\mu^{(1)},\mu^{(2)}),(\tilde \mu^{(1)},\tilde \mu^{(2)}) \in \scr P^{\nu, \mu}_{s,T}.$
 By It\"o's formula and the assumption $({\bf A})$ , we get
\begin{equation}
  \begin{split}
   &\d \big( \big|X_t^{(1),\mu^{(1)}}-X_t^{(1),\tilde \mu^{(1)}}\big|^2+\big|X_t^{(2),\mu^{(2)}}-X_t^{(2),\tilde \mu^{(2)}}\big|^2\big)\\
   &\quad\leq K(t)\Big(|X_t^{(1),\mu^{(1)}}-X_t^{(1),\tilde \mu^{(1)}}\big|^2\\
   &\quad+|X_t^{(2),\mu^{(2)}}-X_t^{(2),\tilde \mu^{(2)}}\big|^2+\W_2(\mu^{(1)}_t, \tilde \mu^{(1)}_t)^2+ \W_2(\mu^{(2)}_t, \tilde \mu^{(2)}_t)^2  \Big)+\d M_t
  \end{split}
\end{equation}
for some martingale $M_t.$
Then, taking expectation on both sides, using Grownwall's inequality and taking $X_s^{(1),\mu^{(1)}}=X_s^{(1),\tilde \mu^{(1)}}=\xi$ and $X_s^{(2),\mu^{(2)}}=X_s^{(2),\tilde \mu^{(2)}}=\eta$ into consideration, we  have
\begin{equation*}\label{esti}
  \begin{split}
   &\e^{-2\lambda t}\big( \E |X_t^{(1),\mu^{(1)}}-X_t^{(1),\tilde \mu^{(1)}}\big|^2+\E|X_t^{(2),\mu^{(2)}}-X_t^{(2),\tilde \mu^{(2)}}\big|^2\big)\\
   &\quad\leq \frac{1}{2\lambda} K(T)\e^{K(T)(T-s)} \sup_{t\in[s,T]} \big[ \e^{-2\lambda t}\big(  \W_2(\mu^{(1)}_t, \tilde \mu^{(1)}_t)^2+ \W_2(\mu^{(2)}_t, \tilde \mu^{(2)}_t)^2\big) \big].
  \end{split}
\end{equation*}
This yields
\begin{equation*}
   \rr_{\lambda}\big(H(\mu^{(1)},~\mu^{(2)}), H(\tilde \mu^{(1)},~\tilde\mu^{(2)})\big)\leq \frac{1}{2\lambda} K(T)\e^{K(T)(T-s)} \rr_{\lambda} \big((\mu^{(1)},~\mu^{(2)}),~(\tilde \mu^{(1)},~\tilde\mu^{(2)})  \big).
\end{equation*}
Hence, for $\lambda>0$ large enough, $H$ is contractive under the metric $\rr_{\lambda}$.
\end{proof}

\begin{proof}[Proof of Theorem \ref{T2}]
Let $s\geq 0$ and $\nu\preceq \mu$ with $\nu_i=\mu_i$. By $ \bar \Lambda_s \preceq \Lambda_s,$
we have $\bar \Lambda_s \nu \preceq \Lambda_s \mu.$ According to \cite[Theorem 5]{ KKO},  there exists $\P_s \in \C(\bar \Lambda_s\nu, \Lambda_s\mu)$ such that
\begin{equation}\label{coup}
\P_s\big(\{(\xi,\eta)\in C_s \times C_s: \xi_t\geq \eta_t, t\geq s\}\big)=1.
\end{equation}
Since $\bar \Lambda_s\nu$ and $ \Lambda_s\mu$ are   solutions to the martingale  problems associated with the operators $\bar L$ and $L$ in \eqref{N2}, respectively, according to the superposition principle (see \cite{SV}),  we have $\L_{(\xi, \eta)}=\P_s$, where $(\xi_t,\eta_t)$  solves
\begin{equation}\label{PSDE1}
\begin{cases}
&\d \xi_t=\bar b(t, \eta_t, \L_{\eta_t}) \d t+\bar \si(t, \eta_t, \L_{\eta_t}) \d B_t^1,~~t\geq s\\
& \d \eta_{t}=b(t,\xi_{t}, \L_{\xi_{t}}) \d t+\si(t,\xi_{t}, \L_{\xi_{t}}) \d B_t^2,~~ t\geq s
\end{cases}
\end{equation}
for
  some $2d$-dimensional Brownian motions $(B_t^1, B_t^2)_{t\geq s}$ on the probability space  $(C_s\times C_s, \B(C_s\times C_s),\{\F_t\}_{t\geq s}, \P_s),$
where $ \{\F_t\}_{t\geq s}$ is the natural filtration induced by the coordinate processes $(\xi_t, \eta_t)_{t\geq s}$.

Since $\L_{(\xi, \eta)}=\P_s$ satisfying \eqref{coup}, then we have $\xi_t\geq \eta_t$  for all $t\geq s.$ Moreover, note that $\L_{(\xi_s, \eta_s)}\in \C(\nu,\mu)$ and $\nu_i=\mu_i$ imply $\xi_s^i=\eta_s^i.$ Thus, we find    $\P_s$-a.s.
\begin{equation}\label{NN3}
  \int_s^t \bar b_i(r, \xi_r,  \mu^{(1)}_{r})\d r+\int_s^t \bar \si_{i\cdot}(r, \xi_r,  \mu^{(1)}_{r})\d B_t^1  \leq   \int_s^t b_i(r, \eta_r,  \mu^{(2)}_{r})+ \int_s^t \si_{i\cdot}(r, \eta_r,  \mu^{(2)}_{r})\d B_t^1,\, t\ge s.
\end{equation}
Taking conditional expectation with respect to $\F_s$, we drive
\begin{equation*}
\int_s^t \E\big(\bar b_i(r, \xi_r,  \mu_r^{(1)})| \F_s\big)\d r \leq \int_s^t \E\big(b_i(r, \eta_r,  \mu_r^{(2)})| \F_s\big)\d r,~~~~t\geq s.
\end{equation*}
 By the continuity of $\bar b$ and $b$ and  $\mu_r^{(1)}\to \nu, \mu_r^{(2)}\to \mu$  weakly as  $r\downarrow s$, we obtain
 \begin{equation*}
\bar b_i(s, \xi_s,  \nu)\le \bar b_i(s, \eta_s,  \nu),~~~t\geq s,~\P_s-a.s.
\end{equation*}
Consequently, for $\pi:=\L_{(\xi_s, \eta_s)}\in \C(\nu, \mu)$ with   $\pi \big(\{x\leq y\}\big)=1$,
$$\bar b_i(s, x, \nu) \leq b_i(s,y, \mu), ~(x,y)\in \supp \pi.$$
Thus, the first assertion of (i) holds true. Hence, for $\nu=\mu,$   $\pi (\{x\leq y\})=1$ implies  $x=y, \pi-a.s.$ Whence, we have
 \begin{equation*}\label{NN4}
 \bar b_i(s, x, \mu) \leq b_i(s,x, \mu), ~x\in \supp \mu.
\end{equation*}
In general, for any $x\in \R^d,$ let $\mu_{\ee}=(1-\ee)\mu+\ee\delta_x.$ It is easy to see that  $x\in \supp\mu_{\ee}.$ Thus applying   \eqref{NN4} with  $\mu_{\ee}$ replaced by  $\mu$ yields
$$\bar b_i(s, x, \mu_\vv) \leq b_i(s, x, \mu_\vv), ~s\geq 0,~\ee>0.$$
Consequently,  the  second assertion in  $(i)$ follows by taking $\ee \downarrow 0$.

Below we assume $\nu\preceq \mu$ with $\nu_{ij}=\mu_{ij}$ so that
$\nu_i=\mu_i,\nu_j=\mu_j.$ Thus,  we deduce from  \eqref{NN3} that for any $\ee \in [0,1],$
\begin{equation*}
 \begin{split}
&\int_s^t \big[\ee \bar b_i(r, \xi_r,  \mu_r^{(1)})+(1-\ee)\bar b_j(r, \xi_r,  \mu_r^{(1)})\big]\d r+\int_s^t \big[\ee \bar \si_{i\cdot}(r, \xi_r,  \mu_r^{(1)})+(1-\ee)\bar \si_{j\cdot}(r, \xi_r,  \mu_r^{(1)})\big]\d B_r^1\\
&\leq \int_s^t \big[\ee \bar b_i(r, \eta_r,  \mu_r^{(2)})+(1-\ee)\bar b_j(r, \eta_r,  \mu_r^{(2)})\big]\d r+\int_s^t \big[\ee \bar \si_{i\cdot}(r, \eta_r,  \mu_r^{(2)})+(1-\ee)\bar \si_{j\cdot}(r, \eta_r,  \mu_r^{(2)})\big]\d B_r^2
 \end{split}
\end{equation*}	
Dividing both side by $\frac{1}{\sqrt{t-s}}$ and letting $t\downarrow s,$ we  find
\begin{equation*}
\begin{split}
 &N\left(
      0, \ee^2 \bar a_{ii}(s,\xi_s, \nu)+2\ee(1-\ee)\bar a_{ij}(s,\xi_s, \nu)+(1-\ee)^2 \bar a_{jj}(s,\xi_s, \nu)\right)\\
&  \le    N\left(
   0, \ee^2 a_{ii}(s,\eta_s, \mu)+2\ee(1-\ee) a_{ij}(s,\eta_s, \mu)+(1-\ee)^2 a_{jj}(s,\eta_s, \mu)\right) .
      \end{split}\end{equation*}
 By the symmetry of centred normal distribution, this further implies
\begin{equation*}
 \begin{split}&\ee^2 \bar a_{ii}(s,\xi_s, \nu)+2\ee(1-\ee)\bar a_{ij}(s,\xi_s, \nu)+(1-\ee)^2 \bar a_{jj}(s,\xi_s, \nu)\\&=\ee^2 a_{ii}(s,\eta_s, \mu)+2\ee(1-\ee) a_{ij}(s,\eta_s, \mu)+(1-\ee)^2 a_{jj}(s,\eta_s, \mu),  \ee \in [0,1]. \end{split}
\end{equation*}	
Consequently, dividing  by $\ee^2$ on both sides yields
 \begin{equation}\label{NN5}
 \bar a_{ij}(s, \xi_s, \nu) = a_{ij}(s,\eta_s, \mu),~\P_s-a.s.,
\end{equation}
which gives  for  $\pi=\L_{(\xi_s, \eta_s)} \in \C(\nu, \mu),$
$$\bar a_{ij}(s, x, \nu)=a_{ij}(s,y,\mu),~~(x,y)\in \supp \pi, ~s\geq 0.$$
Thus, by the approximation trick above, we can obtain the second assertion in (ii).
\end{proof}

 \begin{proof}[Proof of Theorem \ref{T3}]
Due to $\bar P_{s,t}^* \preceq  P_{s,t}^*$, we have $\L_{\bar X_{s,t}}=\bar P_{s,t}^*\nu \preceq  P_{s,t}^*\mu=\L_{X_{s,t}}$ for $\nu\preceq \mu.$ Therefore, in particular for $f(x)=x_i \in \scr U$, we obtain
$$\E(\bar X_{s,t})_i\leq \E(X_{s,t})_i.$$
Since $\nu_i(f)=\mu_i(f)$, we then deduce from \eqref{PSDE1} with $\xi_t$ and $\eta_t$ replaced by $X_{s,t}$ and $\bar X_{s,t}$ in \eqref{ddsde1} that
$$\int_s^t \E(\bar b_i(s,\bar X_{s,r}, \L_{\bar X_{s,r}} ))\d r\leq \int_s^t \E( b_i(s, X_{s,r}, \L_{X_{s,r}}))\d r.$$
Dividing by $t-s$ on both side followed by $t\downarrow s$, we get $(i).$

Since $\bar P_{s,t}^*\mu \leq  P_{s,t}^*\mu,$ for
$f\in \scr U_b\cap C^{\infty}_b(\R^d)$, we have  $\mu(\bar L_{s,\mu}f)\leq \mu( L_{s,\mu}f).$
In particular, taking $\mu=\dd_x$ yields $ \bar L_{s,\dd_x}f(x)\leq   L_{s,\dd_x}f(x).$ With this at hand, we can get the assertion (ii) by following exactly the argument of \cite[Lemma 3.4]{Wang}.
 \end{proof}

\subsection{Proofs of Theorems \ref{T4} and \ref{T5} }
We first present some lemmas. 

\begin{lem}\label{Lem}
Let $\mu=\frac{1}{2} (\mu^{(1)}+\mu^{(2)})$, where $\mu^{(1)}, \mu^{(2)}\in \scr P_+$ such that $\mu^{(1)}\preceq \mu^{(2)}$. Then, $\mu\in\scr P_+.$
\end{lem}
\begin{proof}
Due to  $\mu^{(1)}\preceq \mu^{(2)}$, we have for any $f,g\in\scr U_b$
$$\big(\mu^{(1)}(f)-\mu^{(2)}(f)\big)\big(\mu^{(1)}(g)-\mu^{(2)}(g)\big)\ge0.$$
The above inequality is equivalent to the following inequality 
\begin{equation}\label{E1}
2\big(\mu^{(1)}(f)\mu^{(1)}(g)+\mu^{(2)}(f)\mu^{(2)}(g)\big)\ge \big(\mu^{(1)}(f)+\mu^{(2)}(f)\big) \big(\mu^{(1)}(g)+\mu^{(2)}(g)\big).
\end{equation}
Furthermore,  in terms of $\mu^{(1)}, \mu^{(2)}\in \scr P_+$, we deduce for any $f,g\in\scr U_b$,
$$\mu^{(1)}(fg)\ge \mu^{(1)}(f)\mu^{(1)}(g),\quad\mu^{(2)}(fg)\ge \mu^{(2)}(f)\mu^{(2)}(g).$$ Substituting this \eqref{E1} yields 
\begin{equation*} 
\ff{1}{2}\big(\mu^{(1)}(fg)+\mu^{(2)}(fg) \big)\ge \ff{1}{4}\big(\mu^{(1)}(f)+\mu^{(2)}(f)\big) \big(\mu^{(1)}(g)+\mu^{(2)}(g)\big) 
\end{equation*}
so that   $\mu\in\scr P_+.$
\end{proof}

\begin{lem}\label{lem1}
Suppose that $(P_{s,t}^*)_{t\ge s}$ preserves positive correlations and let $\mu$ be the same as that in Lemma \ref{Lem}. Then, for 
$f,g\in \scr U_b\cap C_b^\infty(\R^d)$ such that
\begin{equation}\label{EE}
\mu(fg)=\mu(f)\mu(g),
\end{equation}
we have
\begin{equation}\label{in1}
\begin{split}
 2 \big(\mu^{(1)}(L_{s, \mu}(fg))+\mu^{(2)}(L_{s, \mu}(fg))\big) &\geq  \big( \mu^{(1)}(L_{s, \mu}(f))+\mu^{(2)}(L_{s, \mu}(f))\big)\big(\mu^{(1)}(g)+\mu^{(2)}(g)\big)\\
 &\quad+\big( \mu^{(1)}(L_{s, \mu}(g))+\mu^{(2)}(L_{s, \mu}(g))\big)\big(\mu^{(1)}(f)+\mu^{(2)}(f)\big).
 \end{split}
 \end{equation}
Equivalently,
\begin{equation}\label{in2}
\begin{split}
 2 \big(\mu^{(1)}(\Gamma_1(f,g))+\mu^{(2)}(\Gamma_1(f,g))\big) &\geq  \big( \mu^{(2)}(L_{s, \mu}(g))-\mu^{(1)}(L_{s, \mu}(g))\big)\big(\mu^{(1)}(f)-\mu^{(2)}(f)\big)\\
 &\quad+\big(\mu^{(2)}(L_{s, \mu}(f))+\mu^{(1)}(L_{s, \mu}(f))\big)\big(\mu^{(1)}(g)-\mu^{(2)}(g)\big),
 \end{split}
 \end{equation}
where
$$\Gamma_1(f,g)(x):=L_{s, \mu}(fg)(x)-f(x)L_{s, \mu}g(x)-g(x)L_{s, \mu}f(x)=\sum_{i,j}a_{i,j}(t,x,\mu)\frac{\pp}{\pp x_i}f(x)\frac{\pp}{\pp x_j}g(x).$$
\end{lem}

\begin{proof}
By a direct calculation, we obtain that \eqref{in1} is equivalent to \eqref{in2}. Therefore, it is sufficient to show that \eqref{in1} holds true. 
From Lemma \ref{Lem}, we have  $\mu\in\scr P_+.$ Since $(P_{s,t}^*)_{t\ge s}$ preserves positive correlations, we deduce for any $f,g\in \scr U_b\cap C_b^\infty(\R^d)$, 
$$(P_{s,t}^*\mu)(fg)\ge (P_{s,t}^*\mu)(f)(P_{s,t}^*\mu)(g),\quad t\ge s. $$
 This, together with $\mu=\frac{1}{2} (\mu^{(1)}+\mu^{(2)})$, yields 
 \begin{equation*}
 2\big((P_{s,t}^*\mu^{(1)})(fg)+(P_{s,t}^*\mu^{(2)})(fg)\big)\ge \big((P_{s,t}^*\mu^{(1)})(f)+(P_{s,t}^*\mu^{(2)})(f)\big)\big((P_{s,t}^*\mu^{(1)})(g)+(P_{s,t}^*\mu^{(2)})(g)\big).
 \end{equation*}
 Thus, by taking \eqref{EE} into consideration, we derive that 
 \begin{equation*}
 \begin{split}
 &\ff{2}{t-s}\big((P_{s,t}^*\mu^{(1)})(fg)+(P_{s,t}^*\mu^{(2)})(fg)-(\mu^{(1)})(fg)+(\mu^{(2)})(fg)\big)\\
 &\ge \ff{1}{t-s}\big\{\big((P_{s,t}^*\mu^{(1)})(f)+(P_{s,t}^*\mu^{(2)})(f)\big)\big((P_{s,t}^*\mu^{(1)})(g)+(P_{s,t}^*\mu^{(2)})(g)\big)\\
 &\quad-
 \big( \mu^{(1)} (f)+ \mu^{(2)} (f)\big)\big( \mu^{(1)} (g)+ \mu^{(2)} (g)\big)\big\}.
 \end{split}
 \end{equation*}
 Consequently, the assertion \eqref{in1} follows by taking $t\downarrow s.$
\end{proof}

\begin{lem}\label{lem4}
If $(P_{s,t}^*)_{t\ge s}$ preserves positive correlation and $\mu\in\scr P_+$ same as in Lemma \ref{Lem}, then   $a_{ij}(s,x,\mu)\geq 0$ for any $s\ge0, x\in\R^d$.
\end{lem}

\begin{proof}
For $\vv>0$ and $x\in\R^d$, let $B_{\ee}(x)=\{y\in \R^d |\, |x-y|\leq \ee\}$, the closed ball centred at the point $x$ with the radius $\ee.$ 
For $f,g\in\scr U_b\cap C_b^\infty(\R^d)$, let $f_{\ee}, g_{\ee}\in \scr U_b\cap C_b^\infty(\R^d)$ with $f_{\ee} \in[f(x)-\vv,f(x)+\vv]$ and $g_{\ee} \in[g(x)-\vv,g(x)+\vv]$ on $B_{\ee}(x)^c$
be the point-wise approximation of $f,g\in\scr U_b\cap C_b^\infty(\R^d)$ (i.e., $f_\vv\to f$ and $g_\vv\to g$ as $\vv\to0$).
Below we assume that $\mu\in\scr P_+$ is given in Lemma \ref{Lem}. 
By Lemma \ref{lem1}, we obtain for $f_{\ee}, g_{\ee}\in \scr U_b\cap C_b^\infty(\R^d)$, 
$$\mu(L_{s,\mu}(f_{\ee}g_{\ee}))(x)\geq \mu(f_{\ee})\mu(L_{s,\mu}g_{\ee})(x)+\mu(g_{\ee})\mu(L_{s,\mu}f_{\ee})(x), ~~ s\geq 0.$$
Then, combining the fact that $f_{\ee}, g_{\ee}\in \scr U_b\cap C_b^\infty(\R^d)$ are  constants on $B_{\ee}(x)^c$, 
we have
\begin{equation}\label{EE2}
\begin{split}
\int_{\R^d} (L_{s,\mu}(f_{\ee}g_{\ee}))(y)\frac{I_{B_{\ee}(x)}(y)(dy)}{\mu(B_{\ee}(x))}&\geq
\mu(f_{\ee})\int_{\R^d}L_{s,\mu}g_{\ee}(y))\frac{I_{B_{\ee}(x)}(y)(dy)}{\mu(B_{\ee}(x))}\\
&\quad+
\mu(g_{\ee})\int_{\R^d}L_{s,\mu}f_{\ee}(y))\frac{I_{B_{\ee}(x)}(y)(dy)}{\mu(B_{\ee}(x))}.
\end{split}
\end{equation}
Observe that 
\begin{equation*}
\mu(f_{\ee})=\int_{B_\vv(x)}(f_\vv(z)-f(x))\mu(\d z)+\int_{B_\vv(x)^c}(f_\vv(z)-f(x))\mu(\d z)+f(x)\rightarrow f(x) 
\end{equation*}
as $\vv\to0$, where the first integral goes to zero since $f_\vv$ is uniformly continuous on $B_\vv(x)$, and the second integral tends to zero due to  $f_\vv\in[f(x)-\vv,f(x)+\vv]$ on $B_\vv(x)^c. $  Similarly, we obtain $\mu(g_\vv)\to g(x)$ as $\vv\to 0.$ Furthermore, 
note that $\frac{\mu(B_{\ee}(x) \cap \cdot)}{\mu(B_{\ee}(x))}$ converges weakly  to $\delta_x$ as $\ee \downarrow0.$ Thus, taking $\vv\downarrow0$ on both sides of \eqref{EE2}  yields 
\begin{align*}
L_{s,\mu}(fg)(x)\geq f(x)L_{s,\mu}g(x)+g(x)L_{s,\mu}f(x).
\end{align*}
Thus, by choosing $f,g\in \scr U_b\cap C_b^\infty(\R^d)$ such that in a neighbourhood of $x$
$$f(z)=z_i,~~~g(z)=z_j,$$
we deduce that  $$x_jb_i(s,x,\mu)+x_ib_j(s,x,\mu)+a_{ij}(s,x,\mu)\geq x_jb_i(s,x,\mu)+x_ib_j(s,x,\mu)$$
so that  $a_{ij}(t,x,\mu)\geq 0.$
\end{proof}

\begin{lem}\label{lem2}
Assume $\mu^{(1)}:=\mu^{(1)}_i\times \mu^{(1)}_{\{i\}^c}$ and  $\mu^{(2)}:=\mu^{(2)}_i\times \mu^{(2)}_{\{i\}^c}$  with $\mu^{(1)}_i= \mu^{(2)}_i$, where $\mu^{(1)}_i$, $\mu^{(2)}_i$, $\mu^{(1)}_{\{i\}^c}$, $ \mu^{(2)}_{\{i\}^c}\in \scr P_+, $ and suppose further that 
$(P_{s,t}^*)_{t\ge s}$ preserves positive correlation. Then, 
 $$\mu^{(1)}( b_i(t,  \cdot, \mu^{(1)}))\leq \mu^{(2)}( b_i(t,  \cdot, \mu^{(2)})).$$
\end{lem}
\begin{proof}
Since $\mu^{(1)}_i$, $\mu^{(2)}_i$, $\mu^{(1)}_{\{i\}^c}$, $ \mu^{(2)}_{\{i\}^c}\in \scr P_+, $ we deduce $\mu^{(1)},\mu^{(2)}\in\scr P_+.$
For given $i$ and $k\neq i$, take $f,g \in \scr U_b\cap C_b^\infty(\R^d)$  such that in a neighbourhood of $x$,$$f(z)=z_i -\int_\R r\mu^{(1)}_i(\d r),~~~~g(z)=\frac{\frac{h(z_k)}{1+h(z_k)}-\mu^{(1)}(\frac{ h}{1+ h})}{\mu^{(2)}(\frac{ h}{1+ h})-\mu^{(1)}(\frac{ h}{1+ h})},$$
where $h\in C^{\infty}(\R;\R_+)$ is an increasing function. Note that 
$$
\mu(fg)=\mu(f)=0,\quad \mu=(\mu^{(1)}+\mu^{(2)})/2.
$$
So applying   Lemma \ref{lem1} gives
$$\int b_i(t,  x, \mu^{(1)})\mu^{(1)}(\d x)\leq \int b_i(t,  x,\mu^{(2)})\mu^{(2)}(\d x).$$
\end{proof}

\begin{proof}[Proof of Theorem \ref{T4}]  Since $\bar P_{s,t}^*\mu$ is  the marginal distributions of    $\LL_s\mu$ at time $t$, 
$  \LL_s\in \scr P_+$ implies  $   P_{s,t}^*\in \scr P_+$ for $t\ge s.$  So, it suffices to prove $  \LL_s\in \scr P_+^s$.
To this end,  we only need to prove that for any $\mu_0\in \scr P_+$ and $T>s\ge 0,$ the marginal distribution $\LL_{s,T}\mu_0$ of $\LL_s\mu_0$ on $C_{s,T}:= C([s,T];\R^d)$ satisfies 
\beq\label{*G} (\LL_{s,T}\mu_0) (FG)\ge (\LL_{s,T}\mu_0) (F)\LL_{s,T}(G)\end{equation}
for any bounded increasing functions $F,G$ on $C_{s,T}.$ 
To achieve this,    let
$$\scr D_+=\big\{\nu\in C([s,T];\scr P_2(\R^d)):\nu_s=\mu_0,~\nu_t\in \scr P_+, ~t\in [s,T]\big\},$$ which is a Polish space under the metric for $\lambda >0:$
\beq\label{C2}
\W_{2,\lambda}(\mu,\nu):=\sup_{t \in [s,T]} (\e^{-\lambda t} \W_2(\mu_t,\nu_t)).
\end{equation}
For $\nu \in C([s,T];\scr P_2(\R^d)),  x\in\R^d$ and $t\in [s,T]$,   let
$$b_t^{\nu}(x)=b_t(x, \nu_t),~~~\si_t^{\nu}(x)=\ss{2 a_t(x, \nu_t)}.$$
For $\nu\in \scr D_+ $, consider the following time-dependent SDE
\beq\label{EQ3}
\d X_t^{\nu}=b_t^{\nu}(X_t^{\nu})\d t+\si_t^{\nu}(X_t^{\nu})\d W_t,\quad t\in[s,T],~~X_s^{\nu}=X_s\sim \mu_0.
\end{equation}

For $\nu\in \scr D_+,$ we define the mapping $\nu\mapsto \Phi(\nu)$ as below, $$(\Phi(\nu))_t:=\scr L_{X_t^{\nu}},\quad t\in[s,T].$$
Under $(1)$ and $(2)$, by $\mu_0\in \scr P_+$ and Theorem \ref{T1.6}, we have $\L_{X_t^{\mu}} \in \scr P_+  $ so that   $\Phi: \D \to \D. $
Below, we assume that $\nu^1, \nu^2\in \scr D.$ By It\"o's formula and \eqref{A1},  it follows that
\begin{equation*}
\begin{split}
\d|X_t^{\nu^1}-X_t^{\nu^2}|^2&=\big\{2\<X_t^{\nu^1}-X_t^{\nu^2}, ~b_t^{\nu^1}(X_t^{\nu^1})-b_t^{\nu^2}(X_t^{\nu^2})\>+\|\si_t^{\nu^1}(X_t^{\nu^1})-\si_t^{\nu^2}(X_t^{\nu^2})\|_{HS}^2 \big\}\d t \\
&\quad+2\<X_t^{\nu^1}-X_t^{\nu^2}, (\si_t^{\nu^1}(X_t^{\nu^1})-\si_t^{\nu^2}(X_t^{\nu^2}))\d  W_t\>\\
&\leq K (|X_t^{\nu^1}-X_t^{\nu^2}|^2+\W_2(\nu_t^1, \nu_t^2)^2)\d t+2\<X_t^{\nu^1}-X_t^{\nu^2},  (\si_t^{\nu^1}(X_t^{\nu^1})-\si_t^{\nu^2}(X_t^{\nu^2}))\d W_t\>.
\end{split}
\end{equation*}
Thus, taking expectation on both side yields,
$$\E|X_t^{\nu^1}-X_t^{\nu^2}|^2\leq \E|X_s^{\nu^1}-X_s^{\nu^2}|+K(T)\int_s^t\big(\E|X_r^{\nu^1}-X_r^{\nu^2}|^2+\W_2(\nu^1_r,\nu^2_r)\big)\d r.$$
Noting that, $X_s^{\nu^1}=X_s^{\nu^2},$ so we have
$$\E|X_t^{\nu^1}-X_t^{\nu^2}|^2\leq K(T)\int_s^t\W_2(\nu^1_r,\nu^2_r)^2\d r+K(T)\int_s^t \E|X_r^{\nu^1}-X_r^{\nu^2}|^2\d r.$$
Then Growall's inequality gives
\begin{equation*}
\E|X_t^{\nu^1}-X_t^{\nu^2}|^2\leq K(T)\e^{K(T)T}\int_s^t\W_2(\nu^1_r,\nu^2_r)^2\d r. 
\end{equation*}
This implies for any $\lambda>0,$
\begin{equation}\label{EQ5}
\begin{split}
\e^{-2\lambda t}\E|X_t^{\nu^1}-X_t^{\nu^2}|^2&\leq K(T)\e^{K(T)T}\int_s^t \e^{-2\lambda(t-r)} \e^{-2\lambda r}\W_2(\nu^1_r, \nu^2_r)^2\d r\\
&\leq \frac{K(T)\e^{K(T)T}}{2\lambda}\W_{2,\lambda}(\nu^1, \nu^2)^2.
\end{split}
\end{equation}
Observe that
\begin{equation*}
\begin{split}\W_{2,\lambda}(\Phi(\nu^1),\Phi(\nu^2))&= \sup_{t \in [s,T]} \big(\e^{-\lambda t} \W_2\big(\Phi(\nu^1))_t,(\Phi(\nu^2))_t\big)\big) \leq \sup_{t\in[s,T]}  \big(\e^{-2\lambda t}\E|X_t^{\nu^1}-X_t^{\nu^2}|^2\big)^{\frac{1}{2}}\\
&\leq \Big(\frac{K(T)\e^{K(T)T}}{2\lambda}\Big)^{\frac{1}{2}}\W_{2,\lambda}(\nu^1, \nu^2),
\end{split}
\end{equation*}
where the last inequality is due to \eqref{EQ5}. Then, by  taking $\lambda=4K(T)\e^{K(T)T},$ we conclude
$$W_{2,\lambda}(\Phi(\nu^1),\Phi(\nu^2))\leq \frac{1}{2}\W_{2,\lambda}(\nu^1,\nu^2).$$
This implies that $\scr D_+\ni \nu \mapsto \Phi(\nu)$ is contractive under the metric $\W_{2,\lambda}.$ Therefore,  the Banach fixed point theorem implies that the mapping $\nu \mapsto \Phi(\nu)$ has a unique fixed point, denoted by $\nu.$ Consequently, we have
$$(\Phi(\nu))_t=\nu_t=\scr L_{X_t^{\nu}}\in\scr P_+, ~~t\in[s,T],$$ so that $\LL_{s,T}:=\L_{(X_t^\nu)_{t\in [s,T]}}$.
Therefore, by applying Theorem \ref{T1.6} to the diffusion process generated by $L_t$ with coefficients $(b^\nu,a^\nu)$, we conclude that the present conditions (1) and (2)  imply \eqref{*G}
as desired. 
 
\end{proof}


\begin{proof}[Proof of Theorem \ref{T5}] 
Consider the decoupled SDE
\begin{equation}
\d X_{s,t}^{x,\mu}=b(t, X_{s,t}^{x,\mu}, P_{s,t}^*\mu)\d t+\si(t, X_{s,t}^{x,\mu}, P_{s,t}^*\mu)\d W_t,~\quad t\ge s, \, ~X_{s,s}^{x,\mu}=x,
\end{equation}
where $P_{s,t}^*\mu$ is the marginal distribution of $\Lambda_s$ at the time $t.$
For $x\in \R^d,$ let $\Lambda_s^{x,\mu}=\L_{X_{[s,\infty)}^{x,\mu}}$. Then, for any $\nu \in \scr P,$
$$\Lambda_s^{\nu,\mu}:=\int_{\R^d} \Lambda_s^{x,\mu}\nu(dx)$$is the law of $X_{[s,\infty)}^{\nu,\mu}$ with initial distribution $\nu.$ Note that 
$$\Lambda_s^{\mu}=\Lambda_s^{\mu,\mu}=\L_{X_{[s,\infty)}^{\mu}}.$$
Since $\Lambda_s^{\mu}\in\scr P_+^s$,  we have
\begin{equation}\label{FKG3}
\Lambda_s^{\mu}(FG)\geq \Lambda_s^{\mu}(F)\Lambda_s^{\mu}(G),~~F,G\in \scr U(C_s), \mu\in\scr P_+.
\end{equation}
For $\gamma\in C_s$, let $F(\gamma)=f(\gamma_s)$ with $0\le f\in \scr U(\R^d).$ Then \eqref{FKG3} becomes
$$\Lambda_s^{\nu, \mu}(G)\geq \Lambda_s^{\mu, \mu}(G),~G\in \scr U(C_s),$$
where $\nu(dx):=\frac{f(x)\mu(dx)}{\mu(f)}.$ That is, $\Lambda_s^{\nu, \mu} \geq \Lambda_s^{\mu, \mu}$.  Then  there exit  $\pi_s \in \C(\Lambda_s^{\nu, \mu},\Lambda_s^{\mu, \mu})$   and   Brownian motions $B_t^1~\mbox{and}~B_t^2$ on $(\Omega, \F_t, \P):=(C_s, \si(\gamma_r: r\in [s,t]), \pi_s)$ such that
\begin{equation}\label{PSDE3}
\begin{cases}
&\d \xi_t= b(t, \xi_t, P_{s,t}^*\mu) \d t+ \si(t, \xi_t, P_{s,t}^*\mu) \d B_t^1,~~\L_{\xi_s}=\nu,~~t\geq s\\
& \d \eta_{t}=b(t,\eta_{t}, P_{s,t}^*\mu) \d t+\si(t,\eta_{t}, P_{s,t}^*\mu) \d B_t^2,~~\L_{\eta_s}=\mu,~~ t\geq s.
\end{cases}
\end{equation}
satisfy $\pi_s(\xi\geq\eta)=1.$

For any increasing $0\le f\in \scr U_b$, which  does not depend on $x_i, x_j,$ we have $\nu:=\frac{f d\mu}{\mu(f)}$ with $\mu=\mu_{\{ij\}}\times\mu_{\{ij\}^c}$ such that  $\nu_{ij}=\mu_{ij}.$ Thus, $\xi_s^i=\eta_s^i,~\xi_s^j=\eta_s^j.$ So,
\begin{equation}\label{PSDE4}
\begin{split}
&\xi_t^k-\eta_t^k= \int_s^t \big( b_k(r, \xi_r, P_{s,r}^*\mu)- b_k(r, \eta_r, P_{s,r}^*\mu)\big) \d r\\
&\quad+ \int_s^t \big\< \si_{k\cdot}(r,\eta_{r}, P_{s,r}^*\mu),  \d B_t^1\big\>-\int_s^t \big\< \si_{k\cdot}(r,\eta_{r}, P_{s,r}^*\mu),  \d B_t^2\big\>\geq 0,\quad k=i,j.
\end{split}
\end{equation}
Thus, by following the argument to derive \eqref{NN5}, we have 
$$a_{ij}(s,\xi_s,\mu)=\si_{i\cdot}(s,\xi_s,\mu)\si_{i\cdot}(s,\xi_s,\mu)=\si_{i\cdot}(s,\eta_s,\mu)\si_{j\cdot}(s,\eta_s,\mu)=a_{ij}(s,\eta_s,\mu).$$
This, together with $\nu=\frac{f d\mu}{\mu(f)}$, $\L_{\xi_s}=\nu$ and $\L_{\eta_s}=\mu$, leads to 
\begin{equation}\label{PSDE5}
\begin{split}
\int f(x) a_{ij}(s,x,\mu)\mu(dx)&= \mu(f)\E a_{ij}(s,\xi_s,\mu)\\
&= \mu(f)\E a_{ij}(s,\eta_s,\mu)=\mu(f)\int a_{ij}(s,x,\mu)\mu(dx).
\end{split}
\end{equation}
Let $g$ be a function such that $$\mu(fg)=\mu(f)\mu(g).$$ Then, for $f(x)=I_{A}(x_k: k\neq i,j)$ with $A\in \scr B(\R^{(d-2})),$
we obtain
$$\E^{\mu}(I_{A}g)=\int _{\R^d} I_{A}(x)g(x)\mu(dx)=\mu(A)\mu(g).$$
Now, by the definition of  conditional expectation we get
$$\E^{\mu}\big( g|x_k:k\neq i,j\big)=\mu(g),$$ which obviously  implies that $g$ depends only on $x_i,x_j$. Thus, \eqref{PSDE5} yields the first assertion.

Dividing by $t-s$ on both side of \eqref{PSDE4} and  taking  $t\to s,$ we get
$$\E b_i(s, \xi_s, \mu)\geq \E b_i(s, \xi_s, \mu).$$
This, together with  $\nu=\frac{f d\mu}{\mu(f)}$, $\L_{\xi_s}=\nu$ and $\L_{\eta_s}=\mu$, leads to the second assertion. \end{proof}

\paragraph{Acknowledgement.} The author would like to thank Professor Feng-Yu Wang for his helpful comments.

\end{document}